\documentclass[12pt,reqno]{amsart}
\usepackage{amsmath,amsfonts,amscd,amssymb,epsf,amsthm,enumerate,parskip,mathrsfs,color}  
 \usepackage[a4paper, margin=1in]{geometry}
\usepackage{times}
\usepackage{comment}
\theoremstyle{definition}
\newtheorem{theorem}{Theorem}[section]
\newtheorem{proposition}[theorem]{Proposition}
\newtheorem{lemma}[theorem]{Lemma}
\newtheorem{corollary}[theorem]{Corollary}

\newtheorem{remark}[theorem]{Remark}
 \numberwithin{equation}{section}
\numberwithin{equation}{section}
\newcommand{\di}{\displaystyle}

\definecolor{Purple}{RGB}{100,0,150}
\definecolor{Green}{RGB}{0,120,20}
\usepackage{cite}
\usepackage[colorlinks,citecolor=blue]{hyperref}

\begin{document}
  
\title[Dedekind Eta Function]{An Elementary Proof of the Transformation Formula for the Dedekind Eta Function}
\author{Ze-Yong Kong and Lee-Peng Teo}
\address{Department of Mathematics, Xiamen University Malaysia\\Jalan Sunsuria, Bandar Sunsuria, 43900, Sepang, Selangor, Malaysia.}
\email{MAT2004438@xmu.edu.my, lpteo@xmu.edu.my}
\begin{abstract}The Dedekind eta function $\eta(\tau)$ is defined by
\[\eta(\tau)=e^{\pi i\tau/12}\prod_{n=1}^{\infty}\left(1-e^{2\pi i n\tau}\right),\quad\text{when}\;\text{Im}\,\tau>0.\]
It plays an important role in  number theory, especially in the theory of modular forms. Its 24$^{\text{th}}$ power, $\eta(\tau)^{24}$, is a modular form of weight 12 for the modular group $\text{PSL}\,(2,\mathbb{Z})$. Up to a constant, $\eta(\tau)^{24}$ is equal to the modular discriminant $\Delta(\tau)$.

In this note, we give an elementary proof of the transformation formula for the Dedekind eta function under the action of the modular group $\text{PSL}\,(2,\mathbb{Z})$. We start by giving a proof of the transformation formula
\[\eta\left(-\frac{1}{\tau}\right)=(-i\tau)^{1/2}\eta(\tau)\]using the Jacobi triple product identity and the Poisson summation formula. Both of these formulas have elementary proofs.

After we   establish some identities for the Dedekind sum,
the transformation formula for $\eta(\tau)$ under the transformation
\[\tau\mapsto \frac{a\tau+b}{c\tau+d}\]induced by a general element  of the modular group $\text{PSL}\,(2,\mathbb{Z})$ is  derived by induction.
\end{abstract}
\maketitle
\section{Introduction}
 The Dedekind eta function is introduced by Dedekind in 1877 and is defined in the upper half plane $\di \mathbb{H}=\left\{\tau\,|\,\text{Im}\,\tau>0\right\}$ by the equation
\begin{align}\label{eq3_31_8}
\eta(\tau)=e^{\pi i\tau/12}\prod_{n=1}^{\infty}\left(1-e^{2\pi i n\tau}\right).
\end{align}It is closely related to the theory of modular forms \cite{Apostol_2}. In this note, we are going to derive the following formula which describes the transformation of $\eta(\tau)$ under a linear fractional transformation defined by an element 
 of the modular group $\Gamma=\text{PSL}\,(2,\mathbb{Z})$. 
\begin{equation}\label{eq8_13_6_2}
\eta\left(\frac{a\tau+b}{c\tau+d}\right)=\exp\left\{\pi i\left(\frac{a+d}{12c}+s(-d,c)\right)\right\}\left\{-i(c\tau+d)\right\}^{1/2}\eta(\tau).
\end{equation}
 Here $\di\begin{bmatrix} a & b\\ c & d\end{bmatrix}$ is an element of $\Gamma$ with $c>0$, $s(-d,c)$ is a Dedekind sum. Eq. \eqref{eq8_13_6_2} is known as the Dedekind functional equation for the Dedekind eta function. One can establish that
the number 
\[\omega(a,b,c,d)=\frac{a+d}{c}+12s(-d,c)\] is an integer. Therefore, the function $f(\tau)=\eta(\tau)^{24}$, which is the  24$^{\text{th}}$ power of $\eta(\tau)$, satisfies
\[f\left(\frac{a\tau+b}{c\tau+d}\right)\frac{1}{(c\tau+d)^{12}}=f(\tau)\quad\text{for all}\;\begin{bmatrix} a & b\\c&d\end{bmatrix}\in\Gamma.\] In other words, $\eta(\tau)^{24}$ is a modular form of weight 12 for the modular group $\Gamma$.

The  Dedekind functional equation \eqref{eq8_13_6_2} was proved using a more general transformation formula of Iseki \cite{Iseki_1957} in the book \cite{Apostol_2}. In the special case where $\begin{bmatrix} a & b\\c& d\end{bmatrix}=\begin{bmatrix}0 & -1\\1& 0\end{bmatrix}$, $s(0,1)=0$, and the formula \eqref{eq8_13_6_2} reduces to
\begin{equation}\label{eq230202_2}\eta\left(-\frac{1}{\tau}\right)=(-i\tau)^{-1/2}\eta(\tau).\end{equation}
This formula has been proved using various methods, such as the contour integral method by Siegel \cite{Siegel} (see also the book \cite{Apostol_2}). A slight drawback of Siegel's method is that it involves a limiting process which needs to be justified using advanced theories. 

In this note, we present a  proof of \eqref{eq230202_2} using elementary methods. We first present the proof of the Jacobi triple product formula
\[\prod_{n=1}^{\infty}\left(1-w^{2n}\right)\left(1+w^{2n-1}z^2\right)\left(1+w^{2n-1}z^{-2}\right)=\sum_{n=-\infty}^{\infty} w^{n^2}z^{2n},\quad |w|<1, z\neq 0,\]following the approach  in \cite{Apostol_1}. Then we derive the Euler pentagonal number formula
\[\prod_{n=1}^{\infty}(1-w^n)=\sum_{n=-\infty}^{\infty}(-1)^n w^{\frac{3n^2-n}{2}},\quad |w|<1,\]
from the Jacobi triple product formula.  The Poisson summation formula is then employed to prove the transformation formula \eqref{eq230202_2}.

It is well known that the modular group $\Gamma$ is generated by the two elements $T=\begin{bmatrix} 1 & 1\\0 & 1\end{bmatrix}$ and $S=\begin{bmatrix}0 & -1\\1& 0\end{bmatrix}$. The transformation of $\eta$ under $T$  is given by \[\eta(\tau+1)=e^{\frac{\pi i}{12}}\eta(\tau),\] which is obvious from its definition. The fact that $\Gamma$ is generated by $T$ and $S$ can be proved by induction on $c$ (see for example \cite{Apostol_2}.) Using this idea, we prove the Dedekind functional equation \eqref{eq8_13_6_2} for general transformation using induction. 
This proof is completely elementary.

The purpose of this work is to give a self-contained elementary proof for the Dedekind functional equation. Therefore, we present in detail the proofs of all the results we need.

\bigskip
\section{Fractional Linear Transformations and the Modular Group}
Let $\widehat{\mathbb{C}}=\mathbb{C}\cup\{\infty\}$  be the extended complex plane. It is well-known that a mapping $w:\widehat{\mathbb{C}}\rightarrow\widehat{\mathbb{C}}$ is analytic and bijective if and only if $w$ is a linear fractional transformation, namely, 
\begin{align}\label{eq3_23_5}w(z)=   \frac{az+b}{cz+d} \end{align} for some 4-tuple $(a, b, c, d)$ with $ad-bc\neq 0$. 

For any nonzero complex number $k$, the 4-tuples $(a, b, c, d)$ and $(ka, kb, kc, kd)$ define the same fractional linear transformation. Therefore, we can normalize $a, b, c, d$ by
$$ad-bc=1,$$ and associate this linear fractional transformation with the two-by-two matrix
\begin{align}\label{eq3_23_6}\begin{bmatrix} a &b\\c & d\end{bmatrix}.\end{align} The set of two-by-two matrices of the form \eqref{eq3_23_6} with $ad-bc=1$ is denoted  by $\text{SL}\,(2,\mathbb{C})$. This is a group under matrix multiplication.

Since $(a, b, c, d)$ and $(-a, -b, -c, -d)$ define the same fractional linear transformation, we can define an equivalence relation on $\text{SL}\,(2,\mathbb{C})$ in the following way. If $A$ and $B$ are in $\text{SL}\,(2,\mathbb{C})$, then $A\sim B$ if and only if $$A=\pm B.$$   The quotient of $\text{SL}\,(2,\mathbb{C})$ by this equivalence relation is denoted by $\text{PSL}\,(2,\mathbb{C})$. Let $I$ be the two-by-two identity matrix. Then $H=\{I, -I\}$ is a normal subgroup of $\text{SL}\,(2,\mathbb{C})$. One can easily see that $$\text{PSL}\,(2,\mathbb{C})=\text{SL}\,(2, \mathbb{C})/H.$$Therefore, $\text{PSL}\,(2,\mathbb{C})$ is also a group, which we call the group of fractional  linear  transformations. The group operation is precisely composition of transformations.

The modular group $\Gamma=\text{PSL}\,(2,\mathbb{Z})$ is the subgroup of $\text{PSL}\,(2, \mathbb{C})$ consists of those elements $\di \begin{bmatrix} a &b\\c & d\end{bmatrix}$, where $a, b, c, d$ are integers and $ad-bc=1$. It is well-known that it is generated by the two elements
\begin{equation}\label{eq230130_6}T=\begin{bmatrix} 1 & 1\\0 & 1\end{bmatrix}\hspace{1cm}
\text{and}\hspace{1cm}S=\begin{bmatrix} 0 & -1\\1 & 0\end{bmatrix},\end{equation}
which describe respectively the linear transformation
\[z\mapsto z+1 \hspace{1cm}\text{and}\hspace{1cm}
 z\mapsto -\frac{1}{z}.\]
\bigskip
\section{Jacobi Triple Product Identity}
In this section, we derive the Jacobi triple product identity following the approach in \cite{Apostol_1}.  
\begin{theorem}[Jacobi Triple Product Identity]
When $w$ and $z$ are complex numbers with $|w|<1$ and $z\neq 0$,
\begin{equation}\label{eq230130_1}\prod_{n=1}^{\infty}\left(1-w^{2n}\right)\left(1+w^{2n-1}z^2\right)\left(1+w^{2n-1}z^{-2}\right)=\sum_{n=-\infty}^{\infty} w^{n^2}z^{2n}.\end{equation}

\end{theorem}
\begin{proof}
When $|w|<1$ and $z\neq 0$, the triple product on the left hand side of \eqref{eq230130_1}  converges absolutely. The sum on the right hand side of \eqref{eq230130_1}  also converges absolutely. 

For $|w|<1$ and $z\neq 0$, define the function $F(w,z)$   by
\[F(w,z)=\prod_{n=1}^{\infty}\left(1-w^{2n}\right)\left(1+w^{2n-1}z^2\right)\left(1+w^{2n-1}z^{-2}\right).\] 
For fixed $w$ with $|w|<1$, $F(w,z)$ can be expanded into a power series in $z$. Since $F(w,z)$ is even in $z$, and $F(w, z)=F(w,z^{-1})$,  the power series of $F(w,z)$ in $z$ has the form
\[F(w,z)=\sum_{n=-\infty}^{\infty}a_n(w)z^{2n},\]
with
\[a_{-n}(w)=a_n(w).\]
Since $F(0,z)=1$, we find that $a_0(0)=1$ and $a_n(0)=0$ if $n\neq 0$.

Now, notice that
\begin{align*}
F(w,wz)&=\prod_{n=1}^{\infty} \left(1-w^{2n}\right)\left(1+w^{2n+1}z^2\right)\left(1+w^{2n-3}z^{-2}\right)\\
&=\frac{1+w^{-1}z^{-2}}{1+wz^2}F(w,z)\\
&=w^{-1}z^{-2}F(w,z).
\end{align*}Therefore,
\begin{align*}
\sum_{n=-\infty}^{\infty}a_n(w)w^{2n}z^{2n}=w^{-1}z^{-2}\sum_{n=-\infty}^{\infty}a_n(w)z^{2n}=\sum_{n=-\infty}^{\infty}a_{n+1}(w)w^{-1}z^{2n}.
\end{align*}This implies that for any integer $n$,
\[a_{n+1}(w)=w^{2n+1}a_n(w).\]
By induction, we find that for $n\geq 1$,
\[a_{-n}(w)=a_n(w)=w^{n^2}a_0(w).\]
Therefore,
\begin{equation}\label{eq230130_3}F(w,z)=a_0(w)\sum_{n=-\infty}^{\infty}w^{n^2}z^{2n}.\end{equation}To prove the theorem, we need to show that $a_0(w)=1$ for all $|w|<1$.
Setting $z=e^{\frac{\pi i}{4}}$ in \eqref{eq230130_3}, we have
\begin{equation}\label{eq230130_2}
\frac{F\left(w, e^{\frac{\pi i}{4}}\right)}{a_0(w)} =\sum_{n=-\infty}^{\infty}w^{n^2}i^n.
\end{equation}Since $i^{2n}=i^{-2n}=(-1)^n$  and $i^{-(2n+1)}=-i^{2n+1}$, we find that the odd terms in the right hand side of \eqref{eq230130_2} cancel, and only the even terms left. This gives
\begin{equation}\label{eq230130_4}
\frac{F\left(w, e^{\frac{\pi i}{4}}\right)}{a_0(w)} =\sum_{n=-\infty}^{\infty}(-1)^nw^{4n^2}.
\end{equation}Setting $z=i$ and replace $w$ with $w^4$ in \eqref{eq230130_3}, we have
\begin{equation}\label{eq230130_5}
\frac{F\left(w^4, i\right)}{a_0(w^4)} =\sum_{n=-\infty}^{\infty}(-1)^nw^{4n^2}.
\end{equation}A comparison of \eqref{eq230130_4} and \eqref{eq230130_5} gives
\[\frac{a_0(w^4)}{a_0(w)}=\frac{F(w^4,i)}{F\left(w, e^{\frac{\pi i}{4}}\right)}.\]
This implies that
\begin{align*}
\frac{a_0(w^4)}{a_0(w)}&=\prod_{n=1}^{\infty}
\frac{(1-w^{8n})(1-w^{8n-4})^2}{(1-w^{2n})(1+iw^{2n-1})(1-iw^{2n-1})}\\
&=\prod_{n=1}^{\infty}
\frac{(1-w^{8n})(1-w^{8n-4})^2}{(1-w^{2n})(1+w^{4n-2})}.
\end{align*}Since every positive integer of the form $4n$ is either of the form $8n$ or of the form $8n-4$, we find that
\[\prod_{n=1}^{\infty}(1-w^{8n})(1-w^{8n-4})=\prod_{n=1}^{\infty}(1-w^{4n}).\]On the other hand,
\[(1-w^{8n-4})=(1-w^{4n-2})(1+w^{4n-2}).\]
Therefore, 
\begin{align*}
\frac{a_0(w^4)}{a_0(w)}=\prod_{n=1}^{\infty}\frac{(1-w^{4n})(1-w^{4n-2})}{1-w^{2n}}.
\end{align*}Since every positive integer of the form $2n$ is either of the form $4n$ or of the form $4n-2$, we find that
\[\prod_{n=1}^{\infty}(1-w^{4n})(1-w^{4n-2})=\prod_{n=1}^{\infty}(1-w^{2n}).\]This implies that
\[a_0(w^4)=a_0(w).\]For for any $w$ with $|w|<1$, 
\[a_0(w)=a_0(w^4)=\cdots=a_0(w^{4k})\] for any positive integer $k$. Since $w^{4k}\rightarrow 0$ when $k\rightarrow\infty$, we find that
\[a_0(w)=a_0(0)=1.\]
This proves that
\[F(w,z)=\sum_{n=-\infty}^{\infty}w^{n^2}z^{2n},\]which completes the proof.
\end{proof}

Notice that the mapping $w=e^{ \pi i \tau}$ maps the upper half plane $\mathbb{H}=\{\text{Im}\,\tau>0\}$ to the unit disc $\mathbb{D}=\{|w|<1\}$. Replacing $w$ by $e^{\pi i \tau}$ and $z$ by $e^{\pi i z}$, the Jacobi triple product identity takes the following form.

\begin{corollary}[Jacobi Triple Product Identity]
For any complex numbers $\tau$ and $z$ with $\text{Im}\,\tau>0$, 
\begin{equation}\label{eq230131_4}
\prod_{n=1}^{\infty}\left(1-e^{2\pi i n\tau}\right)\left(1+e^{\pi i (2n-1)\tau}e^{2\pi i z}\right) \left(1+e^{\pi i (2n-1)\tau}e^{-2\pi i z}\right) =\sum_{n=-\infty}^{\infty}e^{\pi i n^2 \tau }e^{2\pi i n z}.
\end{equation}
\end{corollary}
\section{Poisson Summation Formula}
Poisson summation formula is very useful in the study of number theory. It is a consequence of the theory of Fourier series. In this section, we present the Poisson summation formula and apply it to the Gaussian function.

\begin{theorem}
[Poisson Summation Theorem]
Let $f:\mathbb{R}\rightarrow\mathbb{R}$ be a continuously differentiable function such that
\[\sum_{n=-\infty}^{\infty}f(x+n)\quad\text{and}\quad \sum_{n=-\infty}^{\infty}f'(x+n)\]converge  uniformly on  the closed interval $[0,1]$. Then for any $x\in \mathbb{R}$, 
\begin{equation}\label{eq230131_3}\sum_{n=-\infty}^{\infty} f(x+n)=\sum_{n=-\infty}^{\infty}\widehat{f}(n)e^{2\pi i n x},\end{equation}
where
\[\widehat{f}(n)=\int_{-\infty}^{\infty}f(x)e^{-2\pi i n x}dx.\]
\end{theorem}
\begin{proof}
Define
\begin{equation}\label{eq220925_2}F(x)=\sum_{n=-\infty}^{\infty}f(x+n)\quad\text{and}\quad G(x)=\sum_{n=-\infty}^{\infty}f'(x+n).\end{equation}
Due to the assumption of uniform convergence on $[0,1]$, and the fact that $f$ and $f'$ are continuous, $F$ and $G$ are continuous functions on $[0,1]$. It is   easy to verify that  the series for $F(x)$ and $G(x)$ converge uniformly on any closed and bounded interval,  \[F(x+1)=F(x), \hspace{1cm}G(x+1)=G(x),\] and
\[F'(x)=G(x).\]
In particular, $F$ is also continuously differentiable.
Now, since $F$ is a periodic function of period 1, Dirichlet theorem for Fourier series implies that the Fourier series of $F(x)$ converges to $F(x)$. Namely,
\begin{equation}\label{eq220925_1}\sum_{n=-\infty}^{\infty}f(x+n)=F(x)=\sum_{n=-\infty}^{\infty} c_n e^{2\pi i n x},\end{equation}
where 
\[c_n=\int_0^1F(x)e^{-2\pi i n x}dx.\]
Let us  compute $c_n$ in terms of $f$. 
\[c_n=\int_0^1\sum_{k=-\infty}^{\infty}f(x+k)e^{-2\pi i nx}dx.\]Since the first series in \eqref{eq220925_2} converges uniformly, we can interchange summation and integration to obtain
\[c_n =\sum_{k=-\infty}^{\infty} \int_0^1 f(x+k)e^{-2\pi i n x}dx.\] Making a change of variables $x\mapsto x-k$, we have
\begin{align*}
c_n&=\sum_{k=-\infty}^{\infty} \int_k^{k+1} f(x)e^{-2\pi i n (x-k)}dx\\
&=\sum_{k=-\infty}^{\infty} \int_k^{k+1} f(x)e^{-2\pi i nx}dx\\
&=\int_{-\infty}^{\infty}f(x)e^{-2\pi i nx}dx.
\end{align*}This completes the proof of the theorem.
\end{proof}

We apply the Poisson summation formula to a Gaussian function. First let us verify the uniform convergence of the corresponding series.

\begin{lemma}\label{lemma220925_1}
Let $u$ be a positive number and let $b$ be any real number. Define the function   $f:\mathbb{R}\rightarrow\mathbb{R}$ by
\[f(x)=e^{-2\pi i b x}e^{-\pi u x^2}.\]
Then the two series 
\begin{equation*} F(x)=\sum_{n=-\infty}^{\infty}f(x+n)\quad\text{and}\quad G(x)=\sum_{n=-\infty}^{\infty}f'(x+n)\end{equation*}converge  uniformly on $[0,1]$.
\end{lemma}
\begin{proof}It suffices to consider the case where $b=0$. 
We  prove the uniform convergence for the series $G(x)$  by applying the Weierstrass $M$-test. The proof for the series $F(x)$ is similar.

Notice that when $b=0$,
\[f'(x)=-2\pi u x e^{-\pi ux^2}.\]
For $x\in [0,1]$, when $n\geq 1$, 
\[|f'(x+n)|\leq   2\pi u (n+1) e^{-\pi un^2}\leq 2\pi u (n+1)e^{-\pi u n}.\]When $n\geq 2$, 
\[|f'(x-n)|\leq  2\pi u n e^{-\pi u(n-1)^2}\leq 2\pi u ne^{-\pi u (n-1)}.\]The two series
\[\sum_{n=1}^{\infty}   2\pi u (n+1)e^{-\pi un}
\hspace{1cm} \text{and}\hspace{1cm}\sum_{n=2}^{\infty}2\pi u ne^{-\pi u (n-1)}\] are both convergent. By Weiestrass $M$-test, we conclude the uniform convergence of the series $G(x)$  on the interval $[0,1]$.
\end{proof}

\begin{theorem}
Let $u$ be a positive number. For any real numbers $a$ and $b$,
\begin{equation}\label{eq221024_1}\sum_{n=-\infty}^{\infty} e^{-2\pi i (n+a) b}e^{-\pi u (n+a)^2}=\frac{1}{\sqrt{u}}\sum_{n=-\infty}^{\infty} e^{2\pi i n a}e^{-\pi (n+b)^2/u}.\end{equation}
\end{theorem}
\begin{proof}
Let $f(x)$ be the function defined in Lemma \ref{lemma220925_1}. By the Poisson summation formula \eqref{eq230131_3}, we have
\[\sum_{n=-\infty}^{\infty} e^{-2\pi i (n+a) b}e^{-\pi u (n+a)^2}=\sum_{n=-\infty}^{\infty}\widehat{f}(n) e^{2\pi i n a}.\]
Now we only need to compute $\widehat{f}(n)$.
\begin{align*}
\widehat{f}(n)&=\int_{-\infty}^{\infty}e^{-2\pi i b x}e^{-\pi u x^2}e^{-2\pi i n x}dx\\
 &= e^{-\pi (n+b)^2/u}\int_{-\infty}^{\infty} e^{-\pi u (x+i(n+b)/u)^2}dx.
\end{align*}The integral integrates the function $e^{-\pi u z^2}$ over the closed contour ${\rm Im}\,z=(n+b)/u$. Since the function $e^{-\pi u  z^2}$ is analytic, we can shift the contour of integration to the real line ${\rm Im}\,z=0$.
This gives
\[\int_{-\infty}^{\infty} e^{-\pi u (x+i(n+b)/u)^2}dx=\int_{-\infty}^{\infty} e^{-\pi u x^2}dx=\frac{1}{\sqrt{u}}.\]Therefore,
\[\widehat{f}(n)=\frac{1}{\sqrt{u}} e^{-\pi (n+b)^2/u},\]and
the proof is completed.
\end{proof}

Let
\[D=\left\{(\tau, z, w)\in\mathbb{C}^3\,|\, \text{Im}\,\tau>0\right\}.\]
Notice that  both the series 
\[H_1(\tau,z, w)=\sum_{n=-\infty}^{\infty} e^{-2\pi i (n+z)w}e^{\pi i\tau (n+z)^2}\] and
\[H_2(\tau, z)=(-i\tau)^{-1/2}\sum_{n=-\infty}^{\infty} e^{2\pi i n z}e^{-\pi i (n+w)^2/\tau}\]
converge absolutely and uniformly on any compact subsets of $D$. Hence, both of them define analytic functions on $D$. When $\tau=iu$ with $u>0$ and $z=a$ with $a$ real, $w=b$ with $b$ real,
\eqref{eq221024_1} says that
\[H_1(iu, a, b)=H_2(iu, a, b).\]
By analytic continuation, we find that
\[H_1(\tau, z, w)=H_2(\tau, z, w)\hspace{1cm}\text{for all}\;(\tau, z, w)\in D.\]

\begin{corollary}
For any complex numbers $\tau$, $z$ and $w$ with $\text{Im}\,\tau>0$, 
\begin{equation}\label{eq230131_1}\sum_{n=-\infty}^{\infty} e^{-2\pi i  (n+z)w}e^{\pi i \tau (n+z)^2}=(-i\tau)^{-1/2}\sum_{n=-\infty}^{\infty} e^{2\pi i n z}e^{-\pi i (n+w)^2/\tau}.\end{equation}
\end{corollary}

\bigskip
\section{  Transformation  Defined by the Generators of the Modular Group}
In this section, we consider the transformation of the Dedekind eta function $\eta(\tau)$ under the action of the two generators $T$ and $S$ \eqref{eq230130_6} of the modular group $\text{PSL}\,(2,\mathbb{Z})$.

 For the generator $T$, its power $T^m$ defines the transformation $\tau\mapsto \tau+m$. The transformation of $\eta(\tau)$ under the action of $T^m$ is easily deduced.

 \begin{proposition}\label{prop8_13_4}
 If $\tau\in \mathbb{H}$, and $m$ is an integer, we have
 \[\eta(\tau+m)=\exp\left(\frac{\pi i m}{12}\right)\eta(\tau).\]
 \end{proposition}

 For the transformation of $\eta(\tau)$ under the  generator   $S$, we have the following theorem.
\begin{theorem}\label{transform_eta_S}
When $\text{Im}\,\tau>0$,
\begin{equation}\label{eq230130_7}\eta\left(-\frac{1}{\tau}\right)=(-i\tau)^{1/2}\eta(\tau).\end{equation}
\end{theorem}
There are various methods that can be used to prove this transformation formula. In \cite{Apostol_2}, this is proved using Siegel's method which employs residue calculus.  In \cite{Stein}, this was proved using the Jacobi triple product formula as well as the Poisson summation formula. However, it was first proved that
\[ \eta\left(-\frac{1}{\tau}\right)^3=(-i\tau)^{3/2}\eta(\tau)^3.\]
Here we are going  to give an alternative way to prove \eqref{eq230130_7} by first deriving the Euler pentagonal number theorem.

\begin{theorem}[Euler Pentagonal Number Theorem]\label{thm4_5_1}
When $\text{Im}\,\tau>0$,
\begin{equation}\label{eq230131_2}
\prod_{n=1}^{\infty}\left(1-e^{2\pi i n\tau}\right)=\sum_{n=-\infty}^{\infty}(-1)^ne^{\pi i (3n^2-n)\tau}=\sum_{n=-\infty}^{\infty}(-1)^ne^{\pi i (3n^2+n)\tau}.
\end{equation}

\end{theorem}
\begin{proof}
 Replacing $\tau$ by $3\tau$, and $z$ by $(\tau+1)/2$ in the Jacobi triple product identity \eqref{eq230131_4}, we have
\begin{align*}
\prod_{n=1}^{\infty}\left(1-e^{2\pi i (3n)\tau}\right)\left(1-e^{2\pi i (3n-1)\tau}\right)\left(1-e^{2\pi i (3n-2)\tau}\right)=\sum_{n=-\infty}^{\infty}(-1)^ne^{\pi i (3n^2+n)\tau}.
\end{align*}When $n$ runs through positive integers, $3n$, $3n-1$ and $3n-2$ also runs through positive integers. Hence,
\[\prod_{n=1}^{\infty}\left(1-e^{2\pi i (3n)\tau}\right)\left(1-e^{2\pi i (3n-1)\tau}\right)\left(1-e^{2\pi i (3n-2)\tau}\right)=\prod_{n=1}^{\infty}\left(1-e^{2\pi i n\tau}\right).\]
This completes the proof of \eqref{eq230131_2}.
\end{proof}

Using the Euler pentagonal number theorem, we can express the Dedekind eta function as a sum.
\begin{corollary}
When $\text{Im}\,\tau>0$,
\begin{equation}\label{eq230131_8}
\eta(\tau)=\sum_{n=-\infty}^{\infty}e^{\pi i n}e^{3\pi i\left(n+\frac{1}{6}\right)^2\tau}.
\end{equation}
\end{corollary}
\begin{proof}
This follows directly from the definition of $\eta(\tau)$ \eqref{eq3_31_8} and the Euler pentagonal number theorem \eqref{eq230131_2}.
\end{proof}

Using the identity \eqref{eq230131_1}, we can now prove the transformation formula \eqref{eq230130_7}.

\begin{proof}[Proof of Theorem \ref{transform_eta_S}]
Replacing $\tau$ by $\tau/3$, and set $z=1/2$, $w=1/6$ in \eqref{eq230131_1}, we have
\begin{equation}\label{eq230131_5} \sum_{n=-\infty}^{\infty} e^{\pi i n}e^{-3\pi i  (n+1/6)^2/\tau}=\frac{(-i\tau)^{1/2}}{\sqrt{3}}\sum_{n=-\infty}^{\infty} e^{-\pi i n/3}e^{-\pi i /6}e^{\pi i \tau(n+1/2)^2/3 }.\end{equation}
By \eqref{eq230131_8}, the left hand side of \eqref{eq230131_5} is 
$\eta(-1/\tau)$. For the right hand side, notice that when $n$ runs through all integers, $3n$, $3n-1$ and $3n+1$ together also run through all integers. Therefore,
\begin{equation}\label{eq230131_6}\begin{split}
&\sum_{n=-\infty}^{\infty} e^{-\pi i n/3}e^{-\pi i /6}e^{\pi i \tau(n+1/2)^2/3 }\\&=
\sum_{n=-\infty}^{\infty} e^{-\pi i n }e^{-\pi i /6}e^{\pi i \tau(3n+1/2)^2/3 }+
\sum_{n=-\infty}^{\infty} e^{-\pi i n}e^{\pi i /6}e^{\pi i \tau(3n-1/2)^2/3 }\\&\quad+\sum_{n=-\infty}^{\infty} e^{-\pi i n}e^{-\pi i /2}e^{\pi i \tau(3n+3/2)^2/3 }.\end{split}
\end{equation}Notice that the first two terms on the right hand side of \eqref{eq230131_6} give
\begin{equation}\label{eq230131_7}\begin{split}
&\sum_{n=-\infty}^{\infty} e^{-\pi i n }e^{-\pi i /6}e^{\pi i \tau(3n+1/2)^2/3 }+
\sum_{n=-\infty}^{\infty} e^{-\pi i n}e^{\pi i /6}e^{\pi i \tau(3n-1/2)^2/3 }\\
&=2\cos\frac{\pi}{6}\sum_{n=-\infty}^{\infty}(-1)^n e^{3\pi i \tau(n+1/6)^2}\\
&=\sqrt{3}\eta(\tau).
\end{split}\end{equation}The last term on the right hand side of \eqref{eq230131_6} is
\begin{align*}
I=-i\sum_{n=-\infty}^{\infty} (-1)^ne^{3\pi i \tau(n+1/2)^2}.
\end{align*}When $n$ runs through all integers, $-1-n$ also runs through all integers. We then find that
\begin{align*}
I=-i\sum_{n=-\infty}^{\infty}(-1)^{n+1}e^{3\pi i \tau(-n-1/2)^2}=i\sum_{n=-\infty}^{\infty} (-1)^ne^{3\pi i \tau(n+1/2)^2}=-I.
\end{align*}Thus $I=0$, and we conclude from \eqref{eq230131_5}, \eqref{eq230131_6} and \eqref{eq230131_7} that
\[\eta\left(1/\tau\right)=(-i\tau)^{1/2}\eta(\tau).\]

\end{proof}
 
\begin{remark}
Define the function $\chi:\mathbb{Z}\rightarrow \mathbb{C}$ by
\[\chi(1)=\chi(11)=1,\] \[\chi(5)=\chi(7)=-1, \]\[\chi(2)=\chi(3)=\chi(4)=\chi(6)=\chi(8)=\chi(9)=\chi(10)=\chi(12)=0,\]
and
\[\chi(n+12)=\chi(n)\hspace{1cm}\text{for all}\;n\in\mathbb{Z}.\]
Then $\chi(n)$ is a Dirichlet character modulo 12. It satisfies the multiplicativity property:
\[\chi(mn)=\chi(m)\chi(n)\hspace{1cm}\text{for all}\;m,n\in\mathbb{Z}.\]
As $n$ runs through all integers, $2n$ and $2n+1$ together also runs through all integers.  The formula for $\eta(\tau)$ \eqref{eq230131_8} shows that
\begin{align*}
\eta(\tau)&=\sum_{n=-\infty}^{\infty}e^{\pi i\tau(12n+1)^2/12}-\sum_{n=-\infty}^{\infty}e^{\pi i\tau(12n+7)^2/12}\\
&=\frac{1}{2}\left\{\sum_{n=-\infty}^{\infty}\chi(12n+1)e^{\pi i\tau(12n+1)^2/12}+\sum_{n=-\infty}^{\infty}\chi(12n-1)e^{\pi i\tau(12n-1)^2/12}\right.\\
&\quad+\left.\sum_{n=-\infty}^{\infty}\chi(12n+7)e^{\pi i\tau(12n+7)^2/12}+\sum_{n=-\infty}^{\infty}\chi(12n-7)e^{\pi i\tau(12n-7)^2/12}\right\}\\
&=\frac{1}{2}\sum_{m=1}^{12}\chi(m)\sum_{n=-\infty}^{\infty} e^{\pi i \tau (12n+m)^2/12}\\
&=\frac{1}{2}\sum_{n=-\infty}^{\infty}\chi(n)e^{\pi i \tau n^2/12}.
\end{align*}
\end{remark}
Using this, we can give yet another  proof of Theorem \ref{transform_eta_S}.
\begin{proof}[Second proof of Theorem \ref{transform_eta_S}]
One can easily verify that for all integer $n$,
\[\chi(n)=\frac{1}{\sqrt{12}}\sum_{m=1}^{12}\chi(m)e^{\frac{2\pi i mn}{12}}.\]
Therefore,
\begin{align*}
\eta(-1/\tau)=\frac{1}{2\sqrt{12}}\sum_{m=1}^{12}\chi(m)\sum_{n=-\infty}^{\infty}e^{\frac{2\pi i mn}{12}}e^{-\pi i   n^2/(12\tau)}.
\end{align*}
For each $1\leq m\leq 12$, using \eqref{eq230131_1} with $\tau$ replaced by $12\tau$, $z=m/12$ and $w=0$, we find that
\begin{align*}
\eta(-1/\tau)&=\frac{(-i\tau)^{1/2}}{2}\sum_{m=1}^{12}\chi(m)\sum_{n=-\infty}^{\infty}e^{12\pi i \tau (n+m/12)^2}\\
&=\frac{(-i\tau)^{1/2}}{2}\sum_{m=1}^{12}\chi(m)\sum_{n=-\infty}^{\infty}e^{ \pi i \tau (12n+m)^2/12}\\
&=(-i\tau)^{1/2}\eta(\tau).
\end{align*}This is just a special case of a more general transformation formula for the theta function associated with Dirichlet characters \cite{Montgomery}.
\end{proof}
\section{The Dedekind Sums}
For the transformation formula for $\eta$ under a general element of the modular group, we first define the Dedekind sum.

If $h$ is an integer, $k$ is a positive integer larger than 1, the Dedekind sum $s(h,k)$ is defined as
\begin{gather}\label{eq8_13_1}
s(h,k)=\sum_{r=1}^{k-1}\frac{r}{k}\left(\frac{hr}{k}-\left\lfloor\frac{hr}{k}\right\rfloor-\frac{1}{2}\right).
\end{gather}When $k=1$,   
we define $s(h,1)=0$ for any integer $h$. 

The Dedekind sums have the following properties.
\begin{lemma}\label{lemma8_13_9} Let $k$ be a positive integer, and let $h$ and $h'$ be integers relatively prime to $k$. 
If $h\equiv h'\mod k$, then
\[s(h,k)=s(h',k).\]
\end{lemma}
\begin{proof}
The statement is obvious if $k=1$. If $k>1$, there is an integer $m$ such that
\[h'=km+h.\]
Then for any integer $r$,
\begin{equation*}
\begin{split}
\frac{h'r}{k}-\left\lfloor\frac{h'r}{k}\right\rfloor&=\frac{(km+h)r}{k}-\left\lfloor\frac{(km+h)r}{k}\right\rfloor\\
&=mr+\frac{hr}{k}-\left\lfloor mr+\frac{hr}{k}\right\rfloor\\
&=mr+\frac{hr}{k}-mr-\left\lfloor  \frac{hr}{k}\right\rfloor\\
&=\frac{hr}{k}-\left\lfloor\frac{hr}{k}\right\rfloor.
\end{split}
\end{equation*}It follows from the definition \eqref{eq8_13_1} that
\[s(h',k)=s(h,k).\]
\end{proof}
There is a simple relation between $s(h,k)$ and $s(-h,k)$.

\begin{lemma}\label{lemma8_13_1}
If $k$ is a positive integer, and $h$ is an integer relatively prime to $k$, then
\[ s(-h,k)=-s(h,k).\]
\end{lemma}
\begin{proof}For each $1\leq r\leq k-1$, there is an $r'$ such that $1\leq r'\leq k-1$ and
\[hr\equiv r'\mod k.\]
This implies that  
\[
\frac{hr}{k}-\left\lfloor\frac{hr}{k}\right\rfloor=\frac{r'}{k},
\]
Since
\[-hr\equiv k-r' \mod k,\]
and $1\leq k-r'\leq k-1$, we have
\[
\frac{-hr}{k}-\left\lfloor\frac{-hr}{r}\right\rfloor=\frac{k-r'}{k}=1-\frac{r'}{k}.
\]
Therefore,
\begin{equation*}
\begin{split}
s(-h,k)&=\sum_{r=1}^{k-1}\frac{r}{k}\left(\frac{-hr}{k}-\left\lfloor\frac{-hr}{k}\right\rfloor-\frac{1}{2}\right)\\
&=\sum_{r=1}^{k-1}\frac{r}{k}\left(\frac{1}{2}-\frac{r'}{k} \right)\\
&=-\sum_{r=1}^{k-1}\frac{r}{k}\left(\frac{hr}{k}-\left\lfloor\frac{hr}{k}\right\rfloor -\frac{1}{2}\right)\\
&=-s(h,k).
\end{split}
\end{equation*}

\end{proof}

Next we present some lemmas that will help us to establish the reciprocity relation between $s(h,k)$ and $s(k,h)$ when $h$ and $k$ are positive integers.
\begin{lemma}\label{lemma8_13_2}
If $k$ is a positive integer, $h$ is an integer relative prime to $k$, then
\[\sum_{r=1}^{k-1}\left\lfloor\frac{hr}{k}\right\rfloor=\frac{(h-1)(k-1)}{2}.\]
\end{lemma}
\begin{proof}
As in the proof of Lemma \ref{lemma8_13_1}, for each $1\leq r\leq k-1$, there is an $r'$ such that $1\leq r'\leq k-1$ and
\[hr\equiv r'\mod k.\]
This implies that  
\[
\frac{hr}{k}-\left\lfloor\frac{hr}{k}\right\rfloor=\frac{r'}{k},
\]
As $r$ runs through the integers from 1 to $k-1$, $r'$ also runs through the integers from 1 to $k-1$.
Therefore,
\begin{equation*}
\begin{split}
\sum_{r=1}^{k-1}\left\lfloor\frac{hr}{k}\right\rfloor&=\sum_{r=1}^{k-1}\frac{hr}{k}-\sum_{r'=1}^{k-1}\frac{r'}{k}\\
&=\frac{(h-1)(k-1)}{2}.
\end{split}
\end{equation*}
\end{proof}
\begin{lemma}\label{lemma8_13_3}
 If $h$ and $k$ are positive integers with $(h,k)=1$, then 
 \begin{equation*}
 \sum_{r=1}^{k-1}\left(\left\lfloor\frac{hr}{k}\right\rfloor\right)^2=2hs(k,h)+\frac{(2hk-3h-k+3)(h-1)}{6}.
 \end{equation*}
\end{lemma}
\begin{proof}Using Lemma \ref{lemma8_13_2}, we find that 
\begin{equation*}
\begin{split}
\sum_{r=1}^{k-1}\left(\left\lfloor\frac{hr}{k}\right\rfloor\right)^2&=\sum_{r=1}^{k-1}\left\lfloor\frac{hr}{k}\right\rfloor\left(\left\lfloor\frac{hr}{k}\right\rfloor+1\right)-
\sum_{r=1}^{k-1}\left\lfloor\frac{hr}{k}\right\rfloor\\
&=2\sum_{r=1}^{k-1}\sum_{s=1}^{ \left\lfloor\frac{hr}{k}\right\rfloor}s-\frac{(h-1)(k-1)}{2}.\end{split}
\end{equation*}
Now we consider the lattice points $(r,s)$ with $1\leq r\leq k-1$ and $1\leq s\leq h-1$. Since $h$ and $k$ are relatively prime, none of these points lie on the line $hx=ky$. 
Hence,
\begin{equation*}
\begin{split}
 \sum_{r=1}^{k-1}\sum_{s=1}^{ \left\lfloor\frac{hr}{k}\right\rfloor}s
&= \sum_{\substack{ 1\leq r\leq k-1, \;1\leq s\leq h-1\\ks\leq hr}}s\\
&= \sum_{ 1\leq r\leq k-1, \;1\leq s\leq h-1  }s- \sum_{\substack{ 1\leq r\leq k-1, \;1\leq s\leq h-1\\ hr\leq ks}}s.
\end{split}
\end{equation*}
It follows that
\begin{equation*}
\begin{split}
\sum_{r=1}^{k-1}\left(\left\lfloor\frac{hr}{k}\right\rfloor\right)^2
&= (k-1)(h-1)h -2\sum_{s=1}^{h-1}\sum_{r=1}^{\left\lfloor\frac{ks}{h}\right\rfloor}s-\frac{(h-1)(k-1)}{2}\\
&=\frac{ (k-1)(h-1)(2h-1)}{2}-2\sum_{s=1}^{h-1}s\left\lfloor\frac{ks}{h}\right\rfloor\\
&=2h\sum_{s=1}^{h-1}\frac{s}{h}\left(\frac{ks}{h}-\left\lfloor\frac{ks}{h}\right\rfloor-\frac{1}{2}\right)-2h\sum_{s=1}^{h-1}\frac{s}{h}\left(\frac{ks}{h} -\frac{1}{2}\right)\\&\quad+\frac{ (k-1)(h-1)(2h-1)}{2}\\
&=2hs(k,h)+\frac{(2hk-3h-k+3)(h-1)}{6}.
\end{split}
\end{equation*}
\end{proof}

Now, we can establish the reciprocity law for Dedekind sums. 
\begin{theorem} [Reciprocity Law for Dedekind Sums]\label{thm8_13_8} If $h$ and $k$ are positive integers with $(h,k)=1$, then 
\[s(h,k)+s(k,h)=\frac{h^2+k^2-3hk+1}{12hk}.\]
\end{theorem}
\begin{proof}
Since there is a symmetry in $h$ and $k$, we can assume that $h\geq k$. 
It is easy to check that the formula is true when $h=k=1$. When $k=1$ and $h>1$,
\begin{equation*}
\begin{split}
s(h,k)+s(k,h)&=s(1,h)\\
&=\sum_{r=1}^{h-1}\frac{r}{h}\left(\frac{r}{h}-\frac{1}{2}\right)\\
&=\frac{h^2-3h+2}{12h}\\
&=\frac{h^2+k^2-3hk+1}{12hk}.
\end{split}
\end{equation*}
Now consider the general case where $h\geq k>1$. Notice that since $(h,k)=1$, we must have $h>k$.

As in the proof of Lemma \ref{lemma8_13_1}, for each integer $1\leq r\leq k-1$, there is a unique $r'$ such that $1\leq r'\leq k-1$ and
\[hr\equiv r' \mod k,\]
which implies that
\[\frac{hr}{k}-\left\lfloor\frac{hr}{k}\right\rfloor=\frac{r'}{k}.\]
Hence,
\begin{equation*}
\begin{split}
\sum_{r'=1}^{k-1}\left(\frac{r'}{k}\right)^2&=\sum_{r=1}^{k-1}\left(\frac{hr}{k}-\left\lfloor\frac{hr}{k}\right\rfloor\right)^2\\
&=2\sum_{r=1}^{k-1}\frac{hr}{k}\left(\frac{hr}{k}-\left\lfloor\frac{hr}{k}\right\rfloor-\frac{1}{2}\right)-\sum_{r=1}^{k-1} \frac{h^2r^2}{k^2}+\sum_{r=1}^{k-1}\left(\left\lfloor\frac{hr}{k}\right\rfloor\right)^2+\sum_{r=1}^{k-1} \frac{hr}{k}.
\end{split}
\end{equation*}
Using Lemma \ref{lemma8_13_3}, we find that
\begin{equation*}
\begin{split}
2hs(h,k)&=\sum_{r'=1}^{k-1}\left(\frac{r'}{k}\right)^2+\sum_{r=1}^{k-1} \frac{h^2r^2}{k^2}-\sum_{r=1}^{k-1} \frac{hr}{k}-2hs(k,h)-\frac{(2hk-3h-k+3)(h-1)}{6}
\end{split}
\end{equation*}
Hence,
\begin{equation*}
\begin{split}
s(h,k)+s(k,h)&=\frac{1}{2h}\left\{\sum_{r'=1}^{k-1}\left(\frac{r'}{k}\right)^2+\sum_{r=1}^{k-1} \frac{h^2r^2}{k^2}-\sum_{r=1}^{k-1} \frac{hr}{k} -\frac{(2hk-3h-k+3)(h-1)}{6} \right\}\\
&=\frac{h^2+k^2-3hk+1}{12hk}.
\end{split}
\end{equation*}
\end{proof}

\section{The Dedekind's Functional Equation}
\begin{theorem}
[The Dedekind's Functional Equation]\label{thm4_5_2}~\\
If $\di\begin{bmatrix} a & b\\ c & d\end{bmatrix}\in \Gamma$ and $c>0$, then
\begin{equation}\label{eq8_13_6}
\eta\left(\frac{a\tau+b}{c\tau+d}\right)=\exp\left( \frac{\pi i\omega(a,b,c,d)}{12 } \right)\left\{-i(c\tau+d)\right\}^{1/2}\eta(\tau),
\end{equation}where
\begin{equation}\label{eq230202_1}\omega(a,b,c,d)=\frac{a+d}{c}+12s(-d,c)\end{equation} is an integer.
\end{theorem}
Notice that Theorem \ref{thm4_5_1} is a special case of Theorem \ref{thm4_5_2}. 
\begin{proof}
We prove this by induction on $c$. When $c=1$, $b=ad-1$. Thus,
\[\frac{a\tau+b}{c\tau+d}=\frac{a(\tau+d)-1}{\tau+d}=a-\frac{1}{\tau+d}.\]
It follows from Proposition \ref{prop8_13_4} and Theorem \ref{thm4_5_1} that
\begin{equation*}
\begin{split}
\eta\left(\frac{a\tau+b}{c\tau+d}\right)&=\eta\left(a-\frac{1}{\tau+d}\right)\\
&=\exp\left(\frac{\pi i a}{12}\right)\eta\left(-\frac{1}{\tau+d}\right)\\
&=\exp\left(\frac{\pi i a}{12}\right)\left\{-i(\tau+d)\right\}^{1/2}\eta\left( \tau+d \right)\\
&=\exp\left(\frac{\pi i (a+d)}{12}\right)\left\{-i(\tau+d)\right\}^{1/2}\eta\left( \tau  \right)\\
=&\exp\left( \frac{\pi i\omega(a,b,c,d)}{12 } \right)\left\{-i(c\tau+d)\right\}^{1/2}\eta(\tau),
\end{split}
\end{equation*}where \[\omega(a,b,c,d)=a+d\] is an integer. Since $s(-d,c)=s(-d,1)=0$, this proves \eqref{eq8_13_6} when $c=1$.  

Now we use  principle of strong induction to prove the general case.  Let $c$ be an integer larger than or equal to 2.
Suppose that for all $\di\begin{bmatrix} a' & b'\\ c' & d'\end{bmatrix}\in \Gamma$  with $1\leq c'\leq c-1$,  we have proved the formula \eqref{eq8_13_6} and the statement that $\omega(a', b', c', d')$ is an integer.  Now consider $ \di\begin{bmatrix} a & b\\ c  & d\end{bmatrix}$ with $ad-bc=1$. Since $c$ and $d$ are relatively prime, there is a unique positive integer $r$ less than $c$ such that $-d\equiv r$ mod $ k$.  In other words, there is an integer $q$ such that \[
d=cq-r.\]
Let \[u=aq-b.\] Then
\[\begin{bmatrix} a & b\\ c  & d\end{bmatrix}=\begin{bmatrix} u & a\\ r  & c\end{bmatrix}
\begin{bmatrix} 0 & -1\\1 & 0\end{bmatrix}\begin{bmatrix} 1 & q\\0 & 1\end{bmatrix}.\]
Let $\gamma_1$, $\gamma_2$, $\gamma_3$ be respectively the linear fractional transformations
\[\gamma_1(\tau)=\frac{u\tau +a}{r\tau+c},\quad\gamma_2(\tau)=S(\tau)=-\frac{1}{\tau},\quad \gamma_3(\tau)=T^q(\tau)=\tau+q.\]
Then 
\[\frac{a\tau+b}{c\tau+d}=\gamma_1(\tau')=\frac{u\tau' +a}{r\tau'+c},\quad \tau'=\gamma_2(\gamma_3(\tau))=-\frac{1}{\tau+q}.\]Since $0<r<c$, we can apply induction hypothesis and obtain
\begin{equation*}
\begin{split}
\eta\left(\frac{a\tau+b}{c\tau+d}\right)&=\eta \left(\frac{u\tau' +a}{r\tau'+c}\right)\\
&=\exp\left( \frac{\pi i \omega(u,a,r,c)}{12  } \right) \left\{-i(r\tau'+c)\right\}^{1/2}\eta(\tau'),
\end{split}
\end{equation*}where
\[\omega(u,a,r,c)=\frac{u+c}{r}+12s(-c,r)\] is an integer.
From the case $c'=1$, we  have
\begin{equation*}
\eta(\tau')=\eta\left(-\frac{1}{\tau+q}\right)=\exp\left(\frac{\pi i q}{12}\right)\left\{-i(\tau+q)\right\}^{1/2}\eta(\tau).
\end{equation*} Since
\[(r\tau'+c)(\tau+q)=c(\tau+q)-r=c\tau+d,\]
and \[(-i)^{1/2}=\exp\left(-\frac{\pi i }{4}\right),\]we find that
\[\eta\left(\frac{a\tau+b}{c\tau+d}\right)=\exp\left( \frac{\pi i\omega(a,b,c,d)}{12 } \right) \left\{-i(c\tau+d)\right\}^{1/2}\eta(\tau),\]
where
\begin{equation*}
\omega(a,b,c,d)=\omega(u,a,r,c)+q-3= \frac{u+c+qr}{ r}+12 s(-c,r)-3.
\end{equation*}From the first equality, we conclude by the inductive hypothesis  that $\omega(a,b,c,d)$ is an integer. Now we need to prove that $\omega(a,b,c,d)$ is given by \eqref{eq230202_1}.
By Lemma \ref{lemma8_13_1}, 
\[s(-c,r)=-s(c,r).\]By Theorem \ref{thm8_13_8}, we find that
\begin{equation*}
\begin{split}
s(-c,r)&=s(r,c)-\frac{r^2+c^2-3rc+1}{12rc}.
\end{split}
\end{equation*}Since $-d$ is congruent to $r$ modulo $c$, Lemma \ref{lemma8_13_9} implies that
\[s(-c,r)=s(-d,c)-\frac{r^2+c^2-3rc+1}{12rc}.\]
Hence,
\begin{equation*}
\omega(a,b,c,d)= \Lambda(a,b,c,d)+12s(-d,c),
\end{equation*}
where
\begin{equation*}
\begin{split}
\Lambda(a,b,c,d)&=\frac{u+c+qr}{r}-3-\frac{r^2+c^2-3rc+1}{ rc}\\
&=\frac{uc+ cqr-r^2 -1}{ rc}\\
&=\frac{c(aq-b)-1+dr}{ rc}\\
&=\frac{a(cq-d)+dr}{ rc}\\
&=\frac{a+d}{ c}.
\end{split}
\end{equation*}This proves  \eqref{eq230202_1}. Hence, the theorem is proved.

\end{proof}

\vfill\pagebreak

\bibliographystyle{amsalpha}
\bibliography{ref}
\end{document}